\documentclass[11pt]{amsart}
\usepackage{amssymb,amsmath,amsthm,mathrsfs,color}

\newcommand{\labeq}[1]{\label{eq:#1}}
\newcommand{\refeq}[1]{(\ref{eq:#1})}
\newcommand{\labt}[1]{\label{thm:#1}}
\newcommand{\reft}[1]{Theorem ~\ref{thm:#1}}
\newcommand{\labl}[1]{\label{lemma:#1}}
\newcommand{\refl}[1]{Lemma ~\ref{lemma:#1}}

\newcommand{\labex}[1]{\label{example:#1}}
\newcommand{\refex}[1]{Example ~\ref{example:#1}}

\newcommand{\labd}[1]{\label{definition:#1}}

\numberwithin{equation}{section} \theoremstyle{plain}
\newtheorem{theorem}[equation]{Theorem}
\newtheorem{lemma}[equation]{Lemma}

\newtheorem{corollary}[equation]{Corollary}

\theoremstyle{remark}
\theoremstyle{definition}

\newtheorem{definition}[equation]{Definition}
\newtheorem{example}[equation]{Example}

\newcommand\A{\mathbf{A}}
\newcommand\B{\mathbf{B}}
\newcommand\U{\mathbf{u}}
\newcommand\V{\mathbf{v}}
\newcommand\C{\mathbb{C}}
\newcommand\D{\mathbb{D}}
\newcommand\Cz{\mathbb{C}\lbrack {\bf z}]}
\newcommand\X{\mathscr{X}}

\newcommand\BX{\mathscr{B}(\X)}
\newcommand\Lu{\mathcal{L}_\A(\U)}
\newcommand\Lv{\mathcal{L}_\B(\V)}
\newcommand\s{\mathfrak{L}}

\begin{document}

\markright{Sherwin Kouchekian and Boris Shekhtman}
\title[On simultaneous similarity of families of commuting operators ]{On simultaneous similarity of families of commuting operators}
\author[S. Kouchekian and  B. Shekhtman]{Sherwin Kouchekian and  Boris Shekhtman}
\address{Sherwin Kouchekian\\Department of Mathematics \& Statistics\\
     University of South Florida\\
        Tampa, FL  33620-5700}
\email{skouchekian@usf.edu}
\address{Boris Shekhtman\\ Department of Mathematics\\
     University of South Florida\\
        Tampa, FL  33620-5700}
\email{shekhtma@usf.edu}
\thanks{2000 {\em Mathematics Subject Classification.\/}\ \  Primary:  15A21, 47A16, 47B99;
 \  Secondary:  47L22, 15A30.}

\begin{abstract}
Characterization of simultaneously similarity for commuting $m$-tuples of operators is an open problem even in finite-dimensional spaces; known as ``A wild problem in linear algebra". In this paper we offer a criteria for simultaneously similarity of $m$-tuples of 
$k$-cyclic commuting operators on an arbitrary Banach space. Moreover, we obtain an additional equivalence condition in the case of finite dimensional Banach spaces, which extends the result found in \cite{BS13} for pairs of cyclic commuting matrices. We also present two applications of our results, one in the case of general multiplication operators on Banach spaces of analytic function, and one for  $m$-tuples of commuting square matrices.
\end{abstract}

\maketitle

\section{\bf  Introduction}
Characterization of simultaneously similarity for pairs of commuting square matrices is a central problem in classifying algebras with wild type representations, see \cite{DF}, \cite{Drozd}, and \cite{HGK}. Gelfand and  Ponomorev \cite{GelPv69} showed that  characterization of
simultaneously similarity for  pairs of commuting square matrices would provide a
characterization of simultaneously similarity for  arbitrary pairs of  square matrices.
Since then, the problem of characterization of simultaneously similarity for  pairs
of commuting square matrices has become known as ``A wild problem in linear algebra". In \cite{BS13},  a necessary and sufficient condition for cyclic pairs of commuting square matrices was given in terms of vanishing ideals of polynomials. In this paper we offer an extension of this result to $k$-cyclic commuting $m$-tuples of  operators on Banach spaces.  We being by introducing some notations.

In what follows,  $(\X, \|\cdot\|)$ will always denote a complex Banach space, and $\BX$  the space of all bounded linear operators from $\X$ into $\X.$ An $m$-tuple  $(A_{1},\ldots ,A_{m})$ is a vector in $\left(\BX\right)^m$ and  is called a {\em commuting $m$-tuple} on $\X$ if  $A_iA_j=A_jA_i$ for all  $i,j=1, \dots, m.$ Two commuting $m$-tuples $(A_{1},\ldots ,A_{m})$ and 
$(B_{1},\ldots , B_{m})$ on $\X,$ denoted by $\A$ and $\B$ respectively, are called {\em simultaneously similar} if there exists an invertible operator $S$ in $\BX$ such that $B_{j}=SA_{j}S^{-1}$ for all $ j=1,\ldots ,m$.  Moreover, we let $\Cz:=\mathbb{C} \lbrack z_{1},...,z_{m}]$ stand for the algebra of polynomials in $m$
variables over the complex field $\C.$ Finally, if   $p({\mathbf z})=p(z_{1},...,z_{m})$ is a polynomial in
$\Cz$ and if  $\A= (A_{1},\ldots ,A_{m})$ belongs to  $\left(\BX\right)^m,$ then by $p(\A)$ we mean the operator $p(A_{1},...,A_{m})$ in $\BX.$

This paper is organized as follows. Section 2 contains our main result, \reft{Main}, which provides  an equivalent condition for simultaneously similarity between pairs of $k$-cyclic $m$-tuples of commuting operators on a Banach space. In Section 3, we show that one can offer an additional useful necessary and sufficient condition to \reft{Main} for finite dimensional case, which is not true in general Banach spaces. The provided simple examples in this section show the stark difference between finite and infinite dimensional Banach spaces. Section 4 concludes our paper with two general applications of \reft{Main} and \reft{Finite}.

\section{\bf  Main Result}
To start with, we need the following definition.
\begin{definition}\labd{kCyc}
For a vector $\U =(u_{1}, \ldots ,u_{k})$  in $\X^k$ and an $m$-tuple  $\A=(A_{1},\ldots ,A_{m})$ in $\left(\BX\right)^m$ define
\begin{equation}\labeq{Lu}
\Lu = \left\{\sum_{j=1}^{k}p_{j}(\A)u_{j}:p_{j}\in \Cz\right\}.
\end{equation}
 If $\Lu$ is dense in $\X,$ we call $\U$  a {\em cyclic $k$-tuple for $\A$,} or $\A$ is $k$-cyclic with respect to the $k$-tuple $\U.$  Moreover if $\U$ is understood,  $\A$ is simply called a {\em $k$-cyclic $m$-tuple}. It is clear that $\Lu$ is a (linear) subspace of $\X,$ and our definition extends the standard definition of cyclicity to an $m$-tuple of operators. 
\end{definition}
We are now in the position to state our main result.
\begin{theorem}\labt{Main}
 Let $\A=(A_{1},\ldots ,A_{m})$ be a commuting $m$-tuple on $\X,$ and suppose $\U=(u_{1},...,u_{k})$ is a cyclic $k$-tuple for $\A$.
A commuting $m$-tuple $\B=(B_{1},\ldots ,B_{m})$ on $\X$ is
simultaneously similar to $\A$ if and only if there exists a cyclic $k$-tuple $\V=(v_{1},...,v_{k})$ for $\B$ and a positive constant $c>0$ such that 
\begin{equation}\labeq{Ineq}
c^{-1}\left\Vert \sum_{j=1}^{k}p_{j}(\B)v_{j}\right\Vert \leq \left\Vert
\sum_{j=1}^{k}p_{j}(\A)u_{j}\right\Vert \leq c\left\Vert
\sum_{j=1}^{k}p_{j}(\B)v_{j}\right\Vert  
\end{equation}%
for all polynomials $p_{j}$ in $\Cz$.
\end{theorem}
\begin{proof} 
First, suppose that $\A$ and $\B$ are
simultaneously similar. Thus, there exists an invertible operator $S$ in $\BX$ such that 
$B_{j}=SA_{j}S^{-1}$ for all $ j=1,\ldots ,m$.  Define $v_{j}=Su_{j}$ for $j=1, \dots , k$, and let $\V=(v_{1}, \dots, v_{k}).$  For any $p$ in $\Cz,$ since $\B$ is a commuting $m$-tuple, we have 
\begin{equation*}
p(\A)=p(A_1, \dots, A_m)= p(S^{-1}B_{1}S, \dots, S^{-1}B_{m}S)=S^{-1} p(B_1, \dots, B_m) S=S^{-1} p(\B)S.
\end{equation*}
 Therefore,
\begin{equation}\labeq{Poly}
\sum_{j=1}^{k}p_{j}(\A)u_{j}=\sum_{j=1}^{k} S^{-1 }p_{j}(\B)S u_{j}=S^{-1}\sum_{j=1}^{k}p_{j}(\B)v_{j}  \quad \text{ for all } \quad p_j \in \Cz.
\end{equation}
Recalling the definition \refeq{Lu}, it follows from \refeq{Poly} that $\Lv= S \Lu .$  Since $S$ is onto and $\Lu$ is dense in $\X$ by the assumption, it follows that $\Lv$ is dense in $\X$ as well. Thus,  $\V=(v_{1},...,v_{k})$ is a cyclic $k$-tuple for $B$. Moreover, \refeq{Poly}  also implies that
\begin{equation}\labeq{c}
\left\Vert \sum_{j=1}^{k}p_{j}(\A)u_{j}\right\Vert \leq \left\Vert S^{-1} \right\Vert \left\Vert \sum_{j=1}^{k}p_{j}(\B)v_{j}\right\Vert 
\quad \text{ for all } \quad p_j \in \Cz.
\end{equation}
Applying $S$ from the left to \refeq{Poly} and taking the norm, we obtain 
\begin{equation}\labeq{1c}
\frac{1}{\|S\|}\left\Vert \sum_{j=1}^{k}p_{j}(\B)v_{j}\right\Vert \leq \left\Vert \sum_{j=1}^{k}p_{j}(\A)u_{j}\right\Vert 
\quad \text{ for all } \quad p_j \in \Cz.
\end{equation}
Now inequality \refeq{Ineq} follows from \refeq{c} and \refeq{1c} with $c=\max \{\|S\|, \|S^{-1}\|\}.$ This establishes the proof of the necessity part.

Conversely, suppose that $\B$ is a commuting $m$-tuple, and  let $\V=(v_{1},...,v_{k})$ be a cyclic $k$-tuple 
for $\B$ which satisfies \refeq{Ineq}. Define  $\s$ from $\Lu \subseteq \X$ onto $\Lv \subseteq \X$ as
\begin{equation}\labeq{s}
	\s \Bigl( \sum_{j=1}^k p_j(\A) u_j \Bigr) = \sum_{j=1}^k p_j(\B) v_j   \quad \text{ for all } \quad p_j \in \Cz.
\end{equation}
First we show that $\s$ is well-defined. To see this, suppose that $x \in \Lu$ has two representations 
$x= \sum_{j=1}^k p_j(\A) u_j$ and $x=\sum_{j=1}^k q_j(\A) u_j$ for some $p_j, q_j \in \Cz$ and $1\leq j\leq k.$ In other words, 
$\sum_{j=1}^k (p_j(\A)-q_j(\A)) u_j=0.$ Now the first inequality of \refeq{Ineq} implies that  $\sum_{j=1}^k (p_j(\B)-q_j(\B)) v_j=0;$ or equivalently,   $\sum_{j=1}^k p_j(\B) v_j = \sum_{j=1}^k q_j(\B) v_j.$ In view of the definition for $\s,$ the last equality is equivalent to
$\s \, x= \s (\sum_{j=1}^k p_j(\A) u_j) = \s (\sum_{j=1}^k q_j(\A) u_j).$ Thus,  $\s$ is well-defined.
 
Furthermore,  $\s$ is clearly linear. It follows trivially from the first inequality of \refeq{Ineq} that $\s$ is also bounded on $\Lu$ with $\|\s\| \leq c.$ Since 
$\Lu$ is dense in $\X,$ the bounded linear transformation (BLT) theorem implies that $\s$ has a unique norm-preserving extension to $\X;$ that is, there exists a unique $S$ in $\BX$ such that $Sx = \s x$ for all $x$ in $\Lu$ and $\|S|=\|\s\|.$ We prove that $S$ is the desired operator which provides the simultaneous similarity between $\A$ and $\B.$ To show $S$ is invertible, in view of the open mapping theorem, it suffices to prove that $S$ is a bijection. 

To do this, first let $x \in \Lu.$ Then $x$ has the form $x=  \sum_{j=1}^k p_j(\A) u_j$ for some $p_j$ in $\Cz.$ Now, suppose $x \in \Lu$ such that $Sx=0.$ Using the fact that $Sx=\s x,$ the definition of $\s$ together with  the second inequality of \refeq{Ineq} imply that $x=0.$ Since $\Lu$ is dense in $\X,$ it follows from the continuity of $S$ that $\text{Ker\,}(S)=\{0\},$ where $\text{Ker\,}(S)$ stands for the kernel of $S.$ This proves that $S$ is one-to-one. 

Next, let $y \in \X.$ By density of $\Lv,$ there exists a sequence $\{y_n\} \subseteq \Lv$ such that $y_n \to y$ in $\X;$ that is,
$\lim_{n \to \infty} \|y_n -y\|=0,$ where $y_n = \sum_{j=1}^k p_j^{(n)} (\B) v_j $ for some sequence of polynomials 
$\{p_1^{(n)}, \dots, p_k^{(n)}\}.$ Letting $x_n = \sum_{j=1}^k p_j^{(n)} (\A) u_j ,$ we have from the definition of $\s$ that $\s x_n = y_n$ for all $n \geq 1.$ Since $\{y_n\} $ converges in $\X,$ and hence a Cauchy sequence, the second inequality of \refeq{Ineq} implies that 
$\{x_n\}$ is also a Cauchy sequence in $\X.$ Therefore, there exists an $x \in \X$ such that $x_n \to x$ in $\X.$ Now, it follows from the continuity of $S$ that
 \begin{equation}\labeq{Onto}
 Sx=\lim_{n\to\infty} Sx_n = \lim_{n\to\infty} \s \, x_n = \lim_{n\to\infty} y_n = y.
 \end{equation}
Thus, $S$ is also surjective.

It remains to show $\A$ and $\B$ are simultaneously similar. We start by fixing  $x\in \X$ and use the density of $\Lu$ to get a sequence $\{x_n\} \subseteq \Lu$ such that $x_n \to x$ in $\X,$ where 
$x_n = \sum_{j=1}^k p_j^{(n)} (\A) u_j $ for some sequence of polynomials $\{p_1^{(n)}, \dots, p_k^{(n)}\}$ in $\Cz.$  For each fixed $n \geq1$ and  $1\leq j \leq k,$ let $q_{\ell,j}^{(n)} = z_{\ell}p_j^{(n)},$ where $1 \leq \ell \leq m.$ 
Clearly for each fixed $\ell,$ 
\begin{equation*}
A_{\ell} x_n = A_{\ell} \sum_{j=1}^k p_j^{(n)} (\A) u_j  =  \sum_{j=1}^k A_{\ell}  p_j^{(n)} (\A) u_j 
			=  \sum_{j=1}^k  q_{\ell,j}^{(n)} (\A) u_j .
 \end{equation*}
Thus, it follows from the definitions of $\s$ and $q_{\ell,j}^{(n)}$ that
\begin{equation}\labeq{q}
\s A_{\ell} x_n = \s  \sum_{j=1}^k  q_{\ell,j}^{(n)} (\A) u_j =  \sum_{j=1}^k  q_{\ell,j}^{(n)} (\B) v_j =  
B_{\ell}  \sum_{j=1}^k  p_{j}^{(n)} (\B) v_j = B_{\ell} \s x_n.
 \end{equation}
Finally, the equation \refeq{q}, together with a similar argument as in \refeq{Onto}, implies
 \begin{equation}\labeq{Sim}
 SA_{\ell}x=\lim_{n\to\infty} S A_{\ell} x_n = \lim_{n\to\infty} \s A_{\ell} x_n =\lim_{n\to\infty}  B_{\ell} \s x_n = 
 \lim_{n\to\infty}  B_{\ell} S x_n = B_{\ell} S x.
 \end{equation}
Since $x$ is arbitrary, it follows from \refeq{Sim} that $S A_{\ell} = B_{\ell} S$ for all $\ell = 1, \dots, m.$ This finishes the proof of the theorem.
\end{proof}

We make a couple of observations here. First of all, we note that the positive constants $c^{-1}$ and $c$  in \refeq{Ineq} could have been equivalently substituted by two  arbitrary positive constants $c_1$ and $c_2,$ respectively. The given format, however,  makes \refeq{Ineq}  invariant under the substitution of $(\A,\U)$ with $(\B,\V),$ or vice versa. This also agrees with the statement of  \reft{Main}, which  is symmetric with respect to $\A$ and $\B.$ Next, the proof of necessity part of  \reft{Main} clearly shows that if $\A$ and $\B$ are two simultaneously similar $m$-tuples, then $A$ is $k$-cyclic if and only if $\B$ is $k$-cyclic. For instance, if $\U=(u_1, \dots , u_k)$ is a cyclic $k$-tuple for $\A,$ then $\V=(Su_1, \dots , Su_k)$ is a cyclic $k$-tuple for $\B,$ where $\U$ and $\V$ both satisfy \refeq{Ineq}. In general, however, the converse is not true; that is, there are simultaneously similar $k$-cyclic commuting $m$-tuples $\A$ and $\B,$ where the corresponding cyclic $k$-tuples $\U$ and $\V$ do not satisfy \refeq{Ineq}. At the end of the next section, we provide a simple example by utilizing \reft{Finite}.

\section{\bf  The Finite Dimensional Case}
In this section, we assume $\X$ has a finite dimension. Therefore, it may be assumed without loss of generality  that $\X=\C^N$ for some natural number $N.$  It is also clear that an 
$m$-tuple $\A=(A_{1},\ldots ,A_{m})$ in $(\BX)^m$ is now simply an $m$-tuple of $N\times N$ matrices on $\C^N.$ Our main goal here is  \reft{Finite}, where we  show that one can add another very useful equivalent condition to \reft{Main}. This equivalent condition, however, does not hold in the infinite dimensional case. A  simple counterexample will be provided at the end of this section. We start with a general  lemma.
\begin{lemma}\labl{Finite}
 Suppose $\A=(A_{1},\ldots ,A_{m})$ and  $\B=(B_{1},\ldots ,B_{m})$ are two $m$-tuples of  commuting $N\times N$ matrices on $\C^N.$ If   $\U=(u_{1},...,u_{k})$ and $\V=(v_{1},...,v_{k})$  are two arbitrary vectors in $\left(\C^N\right)^k,$ then the following conditions are equivalent.
\begin{enumerate}
\item[(I)] There exists a positive constant $c>0$ such that
\begin{equation*}
c^{-1}\left\Vert \sum_{j=1}^{k}p_{j}(\B)v_{j}\right\Vert \leq \left\Vert
\sum_{j=1}^{k}p_{j}(\A)u_{j}\right\Vert \leq c\left\Vert
\sum_{j=1}^{k}p_{j}(\B)v_{j}\right\Vert,
\end{equation*}%
for all polynomials $p_{j}$ in $\Cz$.
\item[(II)]
$
\left\{p_1, \dots, p_k \in \Cz : \sum_{j=1}^{k}p_{j}(\A)u_{j} =0 \right\} =
\left\{p_1, \dots, p_k \in \Cz : \sum_{j=1}^{k}p_{j}(\B)v_{j} =0 \right\}.
$
In other words if $p_1, \dots, p_k \in \Cz ,$ then 
\begin{equation*}
 \sum_{j=1}^{k}p_{j}(\A)u_{j} =0 \quad \text{if and only if} \quad \sum_{j=1}^{k}p_{j}(\B)v_{j} =0.
\end{equation*}
\end{enumerate}
\end{lemma}

\begin{proof}
Since all norms on a finite a dimensional space are equivalent, we may assume the norm in (I) is any arbitrary norm on $\C^N.$ Now if $\sum_{j=1}^{k}p_{j}(\A)u_{j} =0$ for some $p_1, \dots, p_k$ in $ \Cz,$ then it follows from the first inequality of (I) that 
$\sum_{j=1}^{k}p_{j}(\B)v_{j} =0,$ and vice versa. Thus (I) implies (II).

To prove (II) implies (I), define $\s$ from $\Lu$ onto $\Lv$ as in \refeq{s}. Using an exact similar argument as in the proof of \reft{Main}, replacing inequality \refeq{Ineq} with the condition (II), one concludes that $\s$ is well defined. Since $\Lu$ and $\Lv$ are now subspaces of $\C^N,$ they are both closed, and thus complete. This implies that $\s$ is a bounded operator from $\Lu$ onto $\Lv$. Moreover if 
$x=\sum_{j=1}^{k}p_{j}(\A)u_{j}$ belongs to $ \Lu$ such that $\s \, x =0;$ or equivalently
$\sum_{j=1}^{k}p_{j}(\B)v_{j}=0$,  then it follows from (II)  that $x=0;$ that is $\text{Ker\,}(S)=\{0\}.$ This shows that $\s$ is also one to one, and thus it is invertible.  Now using a similar argument given in the first part of the proof of \reft{Main}, we obtain the inequalities in (II) with $c=\max \{\|\s\|, \|\s^{-1}\|\}.$
\end{proof}

An immediate consequence of \refl{Finite} when combined with \reft{Main} is the following result for the finite dimensional case.

\begin{theorem}\labt{Finite}
Let $\A=(A_{1},\ldots ,A_{m})$ and  $\B=(B_{1},\ldots ,B_{m})$ be two $m$-tuples of  commuting $N\times N$ matrices on $\C^N.$ 
If $\A$ is $k$-cyclic with respect to the  $k$-tuple $\U=(u_{1},...,u_{k}),$ then the following statements are equivalent.
\begin{enumerate}
\item[(a)] $\B$ is simultaneously similar to $\A.$
\item[(b)] There exists a cyclic $k$-tuple $\V=(v_{1},...,v_{k})$ for $\B$ and  a positive constant $c>0$ such that
\begin{equation*}
c^{-1}\left\Vert \sum_{j=1}^{k}p_{j}(\B)v_{j}\right\Vert \leq \left\Vert
\sum_{j=1}^{k}p_{j}(\A)u_{j}\right\Vert \leq c\left\Vert
\sum_{j=1}^{k}p_{j}(\B)v_{j}\right\Vert,
\end{equation*}%
for all polynomials $p_{j}$ in $\Cz$.
\item[(c)] There exists a cyclic $k$-tuple $\V=(v_{1},...,v_{k})$ for $\B$ such that
\begin{equation*}
\Bigl\{p_1, \dots, p_k \in \Cz : \sum_{j=1}^{k}p_{j}(\A)u_{j} =0 \Bigr\} =
\Bigl\{p_1, \dots, p_k \in \Cz : \sum_{j=1}^{k}p_{j}(\B)v_{j} =0 \Bigr\}.
\end{equation*}
\end{enumerate}
\end{theorem}

As mentioned in the beginning of this section,  condition (c) of \reft{Finite} is not equivalent to simultaneous similarity in infinite dimensional case. Here is a simple example.

\begin{example}\labex{ex1}
Let $(\X,\|\cdot\|)=(H^2,\|\cdot\|_2),$ where $H^2$ denotes the Hardy space of analytic functions on the open unit disk 
$\D=\{z \in \C : |z| <1\}$ defined as (see \cite{PD})
\begin{equation*}
H^2 = \Biggl\{ f : f(z)=\sum_{n=0}^\infty a_n z^n, \ \ z \in \D, \ \ a_n \in \C, \ \  \text{ and } \ \  \|f\|_2^2=\sum_{n=0}^\infty |a_n |^2 < \infty       
\Biggr\}.
\end{equation*}
Let $A=M_z$ be the operator of multiplication by the independent variable $z$ on $H^2$ defined by  $M_z : f(z) \mapsto zf(z).$ Clearly
$A \in \BX.$ Since polynomials are dense in $H^2,$ $A$ is cyclic with the cyclic vector, say, $u=1.$ Next, we set $B=2A=M_{2z}.$ It follows  that $B \in \BX$ and $B$ is also cyclic with respect to the same cyclic vector $u=1.$

Now if $p \in \C[z],$ then it follows from the definition that $p(A)u=p(z).$ Thus, $p(A)u=0$ if and only if $p(z) \equiv 0;$ or equivalently, 
$\{p \in \C[z] : p(A) u=0 \} =\{0\}.$ A similar argument also implies that $\{p \in \C[z] : p(B) u=0 \} =\{0\}.$ Therefore, condition (c) of \reft{Finite} holds. However, $A$ and $B=2A$ are not (simultaneously) similar to each other. For the sake of completeness, we provide a proof here. So suppose that there exists an invertible $S \in \BX$ such that $2A=S A S^{-1.}$ Itteration of the last equality for $n\geq 1$ gives $2^n A^n= S A^n S^{-1}.$ Consequently,
\begin{equation*}
2^n \|A^n\| = \|2^n A^n\| = \|S A^n S^{-1}\| \leq \|S\| \|A^n\| \|S^{-1}\| \quad \text{ for all } n\geq 1.
\end{equation*}
Since $\|A^n\| > 0$ ($n\geq 1$),  the above inequality implies that $\|S\| \|S^{-1}\| \geq 2^n$ for all $n\geq 1,$ which is absurd as both
 $\|S\|$ and $\|S^{-1}\|$ are finite. Therefore, no such $S$ can exist.
\end{example}

We should mention that in the infinite dimensional case, such as in \refex{ex1}, it may very well happen that the condition in part (c) of \reft{Finite} is trivially verified because both sets described in this part
  turn out to be the zero ideal $\{0\}.$ In the finite dimensional case, however, this phenomena can never occur. To see this, recall that by the Cayley-Hamilton theorem every $N\times N$ matrix $A$ satisfies its own characteristic polynomial; that is, there always exists a polynomial $p,$ and hence infinitely many,  such that  $p(A)$ vanishes identically.  This is the underlying reason why the ideal approach in the infinite dimensional case cannot serve as a fruitful strategy. 

We conclude this section with an example of two cyclic (simultaneously) similar operators $A$ and $B$ with a common cyclic $k$-tuple for which the condition \refeq{Ineq} of \reft{Main} does not hold, see the remarks at the end of Section 2.

\begin{example}\labex{2}
Let $\X=\C^2,$ considered  as the usual 2-dimensional vector space over $\C$ with 
$\{e_1,e_2\}= \bigl\{(1,0)^T , (0,1)^T \bigr\}$   as its standard basis. Note also that 
$(e_1,e_2)$ is trivially a cyclic 2-tuple for any $2\times2$-matrix. Consider $A=\begin{pmatrix} 0 & 1 \\ 0 & 0 \end{pmatrix}$ and $B=\begin{pmatrix} -1 & 1 \\ -1 & 1 \end{pmatrix}.$ A routine check shows that $B=SAS^{-1},$ where 
$S=\begin{pmatrix} -1 & 2 \\ -1 & 1 \end{pmatrix}.$ Thus, $A$ and $B$ are  (simultaneously) similar  with a common cyclic $2$-tuple $\U = \V=(e_1,e_2).$ 
Now if we let $p_1(z)=1$ and $p_2(z)=-z,$ then clearly $\sum_{j=1}^{2}p_{j}(A)u_{j} =e_1 - Ae_2=0,$ whereas  
$\sum_{j=1}^{2}p_{j}(B)v_{j} =e_1 - Be_2 \ne 0.$ This shows that condition (c) of \reft{Finite} does not hold. But in  light of \refl{Finite}, conditions (b) and (c) of \reft{Finite} are always equivalent in the finite dimensional case. Thus, \refeq{Ineq} does not hold either. This completes our argument. We make a final remark that the provided trivial example was only possible due to the added equivalence condition in the finite dimensional case, which can be  easily checked. 

\end{example}

\section{\bf  Examples and Illustrations}
We conclude this paper with some applications of the obtained result. As a first application, let $\Omega$ be a non-empty open subset of $\C^m,$ and denote by $(\X(\Omega), \|\cdot\|_\X)$ the Banach space  of holomorphic functions on $\Omega$ for which $\Cz$ is a dense subset. For example, one could consider Dirichlet-type spaces on the $m$-dimensional unit polydisc 
\begin{equation*}
\D^m=\{ (z_1, \dots, z_m) \in \C^m: |z_j| <1, \ \ j=1\dots, m\},
\end{equation*}
which includes the well-known Hardy and  Bergman spaces over the polydiscs and the polydisc algebra $A(\D^m)$, see \cite{WR}. Furthermore  let 
$M_{z_j} $ denote the operator of multiplication by the $j^{th}$-coordinate $z_j$ on $\X(\Omega)$ defined as  
$M_{z_j} : f({\mathbf z}) \mapsto z_j f({\mathbf z}).$ It follows that $\mathbf{M}=(M_{z_j} , \dots, M_{z_m} )$ is a cyclic $m$-tuple of commuting operators on $\X(\Omega)$ with the cyclic vector $1.$ Since $p(\mathbf{M})1= p({\mathbf z})$ for all $p \in \Cz,$ the following result is an immediate consequence of \reft{Main}.
\begin{corollary}
A commuting $m$-tuple  $A=(A_{1},\dots ,A_{m})$ on $\X(\Omega)$ is
simultaneously similar to $\mathbf{M}=(M_{z_1} , \dots, M_{z_m} )$ if and only if $\A$ is cyclic and
there exists a cyclic vector $u$ in $\X(\Omega)$ for $\A$ such that  
\begin{equation*}
c^{-1} \left\Vert    p       \right\|_\X \leq  \left\|p(\A)u \right\|_\X  
 \leq c
\left\Vert    p       \right\|_\X,
\end{equation*}
for some constant $c >0$ and all polynomials $p$ in $\Cz.$
\end{corollary}

Our next and final result is concerned with the finite dimensional case generalizing the idea already presented in \refex{2}. So let 
$\X=\C^N$ with the standard basis
\begin{equation*}
e_1=(1, 0, \dots, 0)^T, \quad e_2=(0, 1, \dots, 0)^T, \quad e_N=(0, 0, \dots, 1)^T. 
\end{equation*}
As noted in \refex{2}, $\U=(e_1, \dots, e_N)$ is clearly a cyclic $m$-tuple for any $m$-tuple of commuting $N\times N$ matrices 
$\A=(A_1, \dots, A_m).$ Therefore, as a consequence of \reft{Finite}, we have the following corollary which provides an answer to a "A wild problem in Linear Algebra" for the cyclic case.

\begin{corollary}
Two $m$-tuples of commuting $N\times N$ matrices $%
\A=(A_{1},\dots ,A_{m})$ and $\B=(B_{1},\dots ,B_{m})$ are simultaneously
similar if and only if there exists a basis $\V=(v_{1},...,v_{N})$ for $\C^N$ such that
\begin{equation*}
\Biggl\{p_1, \dots, p_N \in \Cz : \sum_{j=1}^{N}p_{j}(\A)e_{j} =0 \Biggr\} =
\Biggl\{p_1, \dots, p_N \in \Cz : \sum_{j=1}^{N}p_{j}(\B)v_{j} =0 \Biggr\}.
\end{equation*}

\end{corollary}


\end{document}